\documentclass[11pt,reqno]{amsart}

\usepackage[T1]{fontenc}
\usepackage{enumitem}
\usepackage{lmodern}
\usepackage{amsmath,amssymb,mathtools}
\usepackage{booktabs}
\usepackage{microtype}
\usepackage[colorlinks=true,linkcolor=blue,citecolor=blue,urlcolor=blue]{hyperref}

\newcommand{\Aut}{\operatorname{Aut}}
\newcommand{\SL}{\operatorname{SL}}
\newcommand{\PSL}{\operatorname{PSL}}
\newcommand{\AGL}{\operatorname{AGL}}
\newcommand{\PSU}{\operatorname{PSU}}
\newcommand{\Sp}{\operatorname{Sp}}
\newcommand{\OmegaG}{\operatorname{\Omega}}
\newcommand{\Sym}{\operatorname{S}}
\newcommand{\Alt}{\operatorname{A}}
\newcommand{\Dih}{\operatorname{D}}
\newcommand{\Syl}{\operatorname{Syl}}
\newcommand{\Cen}{\mathbf{C}}
\newcommand{\Nor}{\mathbf{N}}
\newcommand{\Op}{\mathbf{O}}
\newcommand{\M}{\operatorname{M}}

\newtheorem{theorem}{Theorem}[section]
\newtheorem{lemma}[theorem]{Lemma}
\newtheorem{proposition}[theorem]{Proposition}
\newtheorem{corollary}[theorem]{Corollary}
\theoremstyle{definition}
\newtheorem{definition}[theorem]{Definition}

\title[Solvable supplements]{Solvable Supplements to Normalizers of Cyclic $2$-Subgroups}

\author[Qiao]{Shou Hong Qiao}
\address{School of Mathematics and Statistics\\
Guangdong University of Technology\\
Guangzhou 510520\\
P.~R.~China}
\email{qshqsh513@163.com}

\author[Xia]{Binzhou Xia}
\address{School of Mathematics and Statistics\\
The University of Melbourne\\
Parkville, VIC 3010\\
Australia}
\email{binzhoux@unimelb.edu.au}


\begin{document}

\begin{abstract}
Amberg and Kazarin proved that a finite group is solvable if the normalizer of every cyclic subgroup of prime power order has a solvable supplement. We substantially relax this hypothesis by requiring it only for cyclic $2$-subgroups. This condition, denoted by $\mathrm{SSN}_2$, sharply restricts the nonabelian composition factors of the group to the family $\PSL_2(q)$, where $q\geq7$ is a prime power satisfying $q\equiv3\pmod4$. Conversely, this family is precisely the nonabelian finite simple groups that satisfy $\mathrm{SSN}_2$. Consequently, a finite group satisfying $\mathrm{SSN}_2$ is solvable if and only if it has no section isomorphic to one of these groups.
\end{abstract}

\maketitle

\section{Introduction}

Throughout this paper all groups are finite. If a group $G$ can be written as $G=HK$ for subgroups $H$ and $K$, then $K$ is called a \emph{supplement} to $H$ in $G$.

Hall's characterization of solvable groups~\cite{Hall1937Characteristic}, together with later contributions of Wielandt~\cite{Wielandt1958} and Kegel~\cite{Kegel1961}, established deep connections between solvability and the existence of suitable supplements. Baer~\cite{Baer1953} characterized the finite groups in which the centralizer of every element of prime power order has a supplement of prime power order; in particular, all such groups are solvable. This direction was later developed by Berkovich--Kazarin~\cite{BK2005}, who considered groups where the normalizer of every subgroup of prime power order admits a supplement of prime power order. Amberg--Kazarin~\cite{AK2005} substantially weakened these assumptions by requiring only that the normalizer of every cyclic subgroup of prime power order admit a solvable supplement. Even when not forcing solvability, some supplement conditions may instead severely restrict the nonabelian composition factors of the group. For example, Baumeister~\cite{Baumeister1999} proved that if every maximal subgroup of $G$ has a solvable supplement in $G$, then every nonabelian composition factor of $G$ is isomorphic to $\PSL_2(q)$ for some $q\in\{5,7,9,11\}$.

Our starting point is the following theorem of Amberg and Kazarin~\cite{AK2005}.

\begin{theorem}[Amberg--Kazarin~{\cite{AK2005}}]\label{thm:amberg-kazarin}
Let $G$ be a finite group. If the normalizer of every cyclic subgroup of prime power order has a solvable supplement in $G$, then $G$ is solvable.
\end{theorem}

In this paper, we substantially weaken the hypothesis of Theorem~\ref{thm:amberg-kazarin} and investigate what remains of its conclusion. For convenience, we introduce the following property.

\begin{definition}
Let $p$ be a prime. We say that a finite group $G$ has property $\mathrm{SSN}_p$ if, for
every cyclic $p$-subgroup $C\leq G$, its normalizer $\Nor_G(C)$ has a solvable supplement in $G$.
\end{definition}

Thus, Theorem~\ref{thm:amberg-kazarin} says that a finite group is solvable if it satisfies $\mathrm{SSN}_p$ for every prime divisor $p$ of its order. Note that the converse is trivially true: a finite solvable group $G$ has $\mathrm{SSN}_p$ for every prime divisor $p$ of $|G|$, by taking the supplement to be $G$.

Our main result Theorem~\ref{thm:main} shows that, somewhat surprisingly, the single condition $\mathrm{SSN}_2$ already comes remarkably close to the all-primes solvability criterion established in Theorem~\ref{thm:amberg-kazarin}. Although it does not force $G$ to be solvable, it restricts every nonabelian composition factor of $G$ to the family $\PSL_2(q)$ with $q\geq7$ and $q\equiv3\pmod4$.

\begin{theorem}\label{thm:main}
Let $G$ be a finite group with property $\mathrm{SSN}_2$. Then every nonabelian composition factor of $G$ is isomorphic to $\PSL_2(q)$ for some prime power $q\geq7$ such that $q\equiv3\pmod4$.
\end{theorem}

Theorem~\ref{thm:main} yields the following precise solvability criterion.

\begin{corollary}\label{cor:solvability}
Let $G$ be a finite group with property $\mathrm{SSN}_2$. Then the following
are equivalent:
\begin{enumerate}[label=\textup{(\alph*)},leftmargin=*]
\item\label{item:solvability-solvable} $G$ is solvable;
\item\label{item:solvability-composition-factor} $G$ has no composition
factor isomorphic to $\PSL_2(q)$ for any prime power $q\geq7$ such that
$q\equiv3\pmod4$;
\item\label{item:solvability-section} $G$ has no section isomorphic to
$\PSL_2(q)$ for any prime power $q\geq7$ such that $q\equiv3\pmod4$.
\end{enumerate}
\end{corollary}

We note that the choice $p=2$ is structural when considering whether property $\mathrm{SSN}_p$ implies solvability. Indeed, for any fixed odd prime $p$, every nonabelian finite simple group whose order is not divisible by $p$ satisfies $\mathrm{SSN}_p$ vacuously.
Moreover, the next theorem shows that the family $\PSL_2(q)$ appearing in Theorem~\ref{thm:main} and Corollary~\ref{cor:solvability} cannot be narrowed.

\begin{theorem}\label{thm:simple}
Let $S$ be a finite nonabelian simple group. Then $S$ has property $\mathrm{SSN}_2$ if and only if $S\cong\PSL_2(q)$ for some prime power
$q\geq7$ such that $q\equiv3\pmod4$.
\end{theorem}

The paper is organized as follows. In Section~\ref{sec:almost}, we prove the
almost-simple restriction in Proposition~\ref{prop:almost} and the
simple-group classification in Theorem~\ref{thm:simple}. In
Section~\ref{sec:main-proof}, we prove Theorem~\ref{thm:main} and
Corollary~\ref{cor:solvability}.

\section{The almost simple case}\label{sec:almost}

For a prime $p$ and an integer $n\geq1$, denote by $n_p$ the largest $p$-power dividing $n$, and denote $n_{p'}=n/n_p$.
A \emph{factorization} of a group $G$ is an expression $G=HK$ with subgroups $H$ and $K$ of $G$.

\begin{lemma}\label{lem:pprime-factor}
Let $p$ be a prime, and let $G$ be a finite group with a factorization $G=HK$ such that $H$ is solvable and $K$ contains a
Sylow $p$-subgroup of $G$. Then there is a solvable $p'$-subgroup $Q\leq H$ such that $G=QK$.
\end{lemma}

\begin{proof}
Let $D=H\cap K$. Since $K$ contains a Sylow $p$-subgroup of $G$, the index $|H:D|=|G:K|$ is coprime to $p$, and hence $|H|_p=|D|_p$. Since $D\leq H$ is solvable, there exists a Hall $p'$-subgroup $R$ of $D$, and $R$ is contained in some Hall $p'$-subgroup $Q$ of $H$. Now $R\leq D\cap Q$, while $|D\cap Q|\leq |D|_{p'}=|R|$. Hence $D\cap Q=R$, and so
\[
  |QD|=\frac{|Q||D|}{|D\cap Q|}=\frac{|H|_{p'}|D|}{|D|_{p'}}=|H|_{p'}|D|_p=|H|_{p'}|H|_p=|H|.
\]
Thus $H=QD$, and consequently $G=HK=QK$.
\end{proof}

For a finite nonabelian simple group, an \emph{almost simple} group with \emph{socle} $S$ is a group $G$ such that $S\leq G\leq\Aut(S)$.

\begin{lemma}\label{lem:negative}
Let $X$ be an almost simple group with socle $S=\PSL_2(q)$, where $q\geq4$, such that $\Cen_X(t)$ has a solvable supplement in $X$ for
every involution $t\in S$. Then $q\equiv3\pmod4$.
\end{lemma}

\begin{proof}
Choose $P\in\Syl_2(X)$. Since the nontrivial normal subgroup $P\cap S$ of $P$ meets $\mathbf{Z}(P)$ nontrivially, we may choose an involution $z\in \mathbf{Z}(P)\cap S$. Let
\[
  K=\Cen_X(z).
\]
Then $P\leq K$. By the hypothesis of the lemma, $K$ has a solvable supplement in $X$. Then Lemma~\ref{lem:pprime-factor} gives a solvable subgroup $Q$ of odd order such that
\[
  X=KQ.
\]
To prove the lemma, we exclude the cases when $q$ is even and when $q\equiv1\pmod4$ respectively.

Suppose first that $q=2^f$ is even. Then $\Cen_S(z)$ is elementary abelian of order $q$. Since $Q$ has odd order, the $\PSL_2(q)$ case of~\cite[Theorem~1.1]{LX2022} gives $|Q\cap S|\leq q+1$. Moreover, $|Q:Q\cap S|\leq|\operatorname{Out}(S)|=f$. Hence
\[
  q^2-1=|S:\Cen_S(z)|=|S:K\cap S|\leq|X:K|\leq|Q|\leq(q+1)f<(q+1)(q-1),
\]
a contradiction.

Suppose next that $q\equiv1\pmod4$. Then $K\cap S=\Cen_S(z)\cong\Dih_{q-1}$. Since $Q$ has odd order, the factorization $X=QK$ does not occur in the classification in~\cite[Theorem~1.1]{LX2022}, again a contradiction.
\end{proof}

The key result of this section is as follows.

\begin{proposition}\label{prop:almost}
Let $X$ be an almost simple group with socle $S$ such that $\Cen_X(t)$ has a solvable supplement in $X$ for every involution $t\in S$. Then $S\cong\PSL_2(q)$ for some prime power $q\geq7$ with $q\equiv3\pmod4$.
\end{proposition}

\begin{proof}
By Lemma~\ref{lem:negative}, it suffices to prove that $S\cong\PSL_2(q)$ for some prime power $q\geq4$.
Take $P\in\Syl_2(X)$. Since $S\trianglelefteq X$, the subgroup $P\cap S$ is a Sylow $2$-subgroup of $S$, which is nontrivial by the Feit--Thompson odd-order theorem~\cite{FT1963}. Moreover, $P\cap S\trianglelefteq P$, and every nontrivial normal subgroup of a finite \(p\)-group meets its center nontrivially. Hence $\mathbf{Z}(P)\cap S\neq1$. Choose an involution $z\in \mathbf{Z}(P)\cap S$, and let
\[
  K=\Cen_X(z).
\]
Then $P\leq K$. By the hypothesis of the proposition, there is a solvable subgroup $H\leq X$ such that $X=HK$. This together with Lemma~\ref{lem:pprime-factor}, applied with $p=2$, gives a subgroup $Q$ of odd order in $H$ such that
\begin{equation}\label{eq:KQ}
  X=QK.
\end{equation}
We may then apply~\cite[Theorem~1.1]{LX2022} to the factorization~\eqref{eq:KQ}, with solvable factor $Q$ and core-free factor $K$.

Since $z\in \mathbf{Z}(P)$, the Sylow $2$-subgroup $P\cap S$ of $S$ is contained in $\Cen_S(z)$, and so $|S:\Cen_S(z)|$ is odd. Since $z$ is central in $\Cen_S(z)$, the largest normal $2$-subgroup of $\Cen_S(z)$ contains $\langle z\rangle$. Hence we have
\begin{equation}\label{eq:local}
  K\cap S=\Cen_S(z),\ \
  |S:\Cen_S(z)|_2=1,\ \
  1\neq\langle z\rangle\leq \mathbf{Z}(\Cen_S(z))\cap\Op_2(\Cen_S(z)).
\end{equation}
Moreover, for each $K_0\leq\Cen_S(z)$,
\begin{equation}\label{eq:centralizes-K0}
  z\in\Cen_S(K_0).
\end{equation}
We use~\eqref{eq:local}--\eqref{eq:centralizes-K0} as filters to read off candidates for the factorization~\eqref{eq:KQ} from~\cite[Theorem~1.1]{LX2022}. We follow the classification order in~\cite[Theorem~1.1]{LX2022}.
Thus the cases below are exhaustive.

\smallskip
\textsf{Case~1: both factors solvable.}
The factorizations $X=QK$ with both $Q$ and $K$ solvable are listed in Table~4.1 of~\cite{LX2022}.
Since $Q$ has odd order, besides the family $\PSL_2(q)$, this leaves only rows~4--5, 7--8, and~12 of that table, and hence the four possible socles $S$:
\[
  \PSL_3(3),\ \ \PSL_3(8),\ \ \PSU_3(8),\ \ \M_{11}.
\]
In each of these rows,~\cite[Table~4.1]{LX2022} in particular gives a required divisor $d$ of $|K\cap S|$. However, a direct computation in \textsc{Magma}~\cite{BCP1997} gives $|\Cen_S(z)|$, recorded below, which is not divisible by $d$. This contradicts the condition $K\cap S=\Cen_S(z)$ in~\eqref{eq:local}.
\[
\begin{array}{cccc}
\toprule
S&\text{Rows}&|\Cen_S(z)|&\text{Required divisor }d\\
\midrule
\PSL_3(3)&4,5&2^4\cdot3&3^2\\
\PSL_3(8)&7&2^9\cdot7&2^9\cdot7^2\\
\PSU_3(8)&8&2^9\cdot3&7\\
\M_{11}&12&2^4\cdot3&3^2
\\
\bottomrule
\end{array}
\]

\smallskip
\textsf{Case~2: alternating socles.}
The candidates for $X=QK$ are listed in~\cite[Proposition~4.3]{LX2022}.
The groups $\Alt_5$ and $\Alt_6$ are isomorphic to $\PSL_2(4)$ and $\PSL_2(9)$, respectively. Let $S=\Alt_n$ with $n\geq7$. In the infinite families of~\cite[Proposition~4.3]{LX2022}, a natural subgroup $\Alt_{n-k}$ is contained in $K\cap S$ for some $k\leq3$. Therefore, $z$ centralizes $\Alt_{n-k}$. Since $n-k\geq4$, it follows that $\Cen_{\Alt_n}(\Alt_{n-k})=\Alt_k$ has odd order for $k\leq3$. This is incompatible with the involution $z\in S$, excluding the infinite families of~\cite[Proposition~4.3]{LX2022}. Now the only remaining candidate in~\cite[Proposition~4.3]{LX2022} has $S=\Alt_8$ and $K\cong\AGL_3(2)$. However, here $K$ has no nontrivial central $2$-element, contradicting $z\in \mathbf{Z}(K)$.

\smallskip
\textsf{Case~3: sporadic socles.}
Here the candidates are listed in~\cite[Proposition~4.4]{LX2022}. However, $\mathbf{Z}(K)=1$ holds for all the possibilities of $K$, contradicting $z\in \mathbf{Z}(K)$.

\smallskip
\textsf{Case~4: the infinite families for classical socles.}
This is the case where $(S,Q\cap S,K\cap S)$ lies in Table~1.1 of~\cite{LX2022}. Let $M=\Nor_S(K\cap S)$, the maximal
subgroup specified in~\cite[Remark~1.2(i)]{LX2022}. Since $K\cap S\leq M$, the index $|S:M|$ divides $|S:K\cap S|$, and hence is odd by~\eqref{eq:local}. This excludes rows~2--6 and~9 of the table.
Let $K_0$ be the subgroup in the ``$K\cap L\trianglerighteq$'' column of Table~1.1. In row~1 with $n\geq3$, and also in row~8, the fact $\Cen_S(K_0)=1$ contradicts $1\neq z\in\Cen_S(K_0)$ from~\eqref{eq:centralizes-K0}. The case $n=2$ in row~1 is precisely the family $\PSL_2(q)$. It remains only to exclude row~7.

In this row, $S=\OmegaG_{2m+1}(q)$ and $K_0=\OmegaG^-_{2m}(q)$, where $m\geq3$ and $q$ is odd. Choose a nondegenerate $2$-space of suitable spinor type and let
\[
  y=-I_2\oplus I_{2m-1}\in S
\]
corresponding to this $2$-space and its orthogonal complement. Then, for a suitable sign $\varepsilon$,
\begin{equation}\label{eq:orthogonal-r-parts}
  \Cen_S(y)=\bigl(\mathrm O_2^\varepsilon(q)\times\mathrm O_{2m-1}(q)\bigr)\cap S.
\end{equation}
By the hypothesis of the proposition, there is a solvable subgroup $L\leq X$ such that
\[
X=L\Cen_X(y).
\]
However,~\cite[Theorem~1.1]{LX2022} shows that there is no such factorization with $\Cen_X(y)\cap S=\Cen_S(y)$ satisfying~\eqref{eq:orthogonal-r-parts}, a contradiction.

\smallskip
\textsf{Case~5: the finite table for classical socles.}
This is the case where $(S,Q\cap S,K\cap S)$ lies in Table~1.2 of~\cite{LX2022}.
Rows~1--5 of the table already have socle $\PSL_2(q)$. By~\eqref{eq:local}, the index $|S:K\cap S|$ is odd.
This restricts the candidates to be as listed below, each excluded by the recorded contradiction.

\begin{center}
\small
\begin{tabular}{lll}
\toprule
Rows & Possibilities for $K\cap S$ & Contradiction to \eqref{eq:local} \\
\midrule
9, 10, 11
& $2^6{:}(\Sym_3\times\PSL_3(2))$ or $2^4{:}\Alt_5$
& $\mathbf{Z}(K\cap S)\cap\Op_2(K\cap S)=1$ \\
21
& $2^{12}{:}\SL_2(64){.}7$
& $\mathbf{Z}(K\cap S)\cap\Op_2(K\cap S)=1$ \\
23, 28
& $\Sp_6(2)$ or $\OmegaG^+_8(2)$
& $\Op_2(K\cap S)=1$. \\
\bottomrule
\end{tabular}
\end{center}
\end{proof}

\begin{proof}[\rm\textbf{Proof of Theorem~\ref{thm:simple}}]
Suppose that $S$ has property $\mathrm{SSN}_2$. For every involution $t\in S$, we have $\Nor_S(\langle t\rangle)=\Cen_S(t)$. Hence Proposition~\ref{prop:almost}, applied with $X=S$, gives $S\cong\PSL_2(q)$ for some prime power $q\geq7$ with $q\equiv3\pmod4$.

Conversely, let $S=\PSL_2(q)$, where $q\geq7$ and $q\equiv3\pmod4$. To verify that $S$ has property $\mathrm{SSN}_2$, it suffices to consider a nontrivial cyclic $2$-subgroup $H=\langle h\rangle$ of $S$. We use the standard description of subgroups of $S$; see~\cite[Chapter~II, Section~8]{Huppert2025}. The nonsplit maximal torus $T$ of $S$ is a cyclic group of order $(q+1)/2$, such that $M\coloneqq\Nor_S(T)$ is a dihedral group of order $q+1$. Since $|M|_2=|S|_2$, we may assume that, up to conjugation in $S$, the subgroup $H$ is contained in $M$ and is further contained in $T$.
Thus, $\Nor_S(H)$ is contained in $M$ but not in any proper subgroup of $S$ that properly contains $M$. Consequently,
\[
\Nor_S(H)=M.
\]
Let $K$ be a Borel subgroup of $S$. Then $K$ is solvable and $|K|=q(q-1)/2$. From
\[
  \gcd\left(q+1,\frac{q(q-1)}2\right)=1,
\]
we deduce that $M\cap K=1$, and so $|MK|=|M||K|=(q+1)q(q-1)/2=|S|$.
It follows that $S=MK$, meaning that $\Nor_S(H)=M$ has a solvable supplement $K$ in $S$. This completes the proof.
\end{proof}

\section{Proof of Theorem~\ref{thm:main} and Corollary~\ref{cor:solvability}}\label{sec:main-proof}

We begin with two reduction lemmas needed for the induction.

\begin{lemma}\label{lem:quotient}
For every prime $p$, property $\mathrm{SSN}_p$ is inherited by quotient groups.
\end{lemma}

\begin{proof}
Let $G$ be a finite group with property $\mathrm{SSN}_p$, let $N$ be a normal subgroup of $G$, and let $\overline{\phantom{w}}\colon G\to G/N$ be the quotient homomorphism.
Take an arbitrary cyclic $p$-subgroup $\langle\overline g\rangle$ of $\overline G$. Let $K$ be the full pre-image of $\langle\overline g\rangle$ in $G$, and let $P\in\Syl_p(K)$. Then $\overline P$ is a Sylow $p$-subgroup of $\overline K=\langle\overline g\rangle$, and so $\overline P=\langle\overline g\rangle$. Hence there exists $c\in P$ with $\overline c=\overline g$.
Since $G$ has property $\mathrm{SSN}_p$, there exists a solvable subgroup $H\leq G$ such that $G=\Nor_G(\langle c\rangle)H$. It follows that
\[
\overline G=\overline{\Nor_G(\langle c\rangle)}\,\overline H
\leq\Nor_{\overline G}(\langle\overline c\rangle)\overline H
=\Nor_{\overline G}(\langle\overline g\rangle)\overline H\leq\overline G.
\]
Thus, $\overline G=\Nor_{\overline G}(\langle\overline g\rangle)\overline H$, meaning that $\Nor_{\overline G}(\langle\overline g\rangle)$ has a solvable supplement $\overline H$.
\end{proof}

\begin{lemma}\label{lem:minimal-normal}
Let $p$ be a prime, let $G$ be a finite group with property $\mathrm{SSN}_p$ and a minimal normal subgroup $M=S_1\times\cdots\times S_r$, where $S_1\cong\cdots\cong S_r$ is nonabelian simple, and let
\[
N=\Nor_G(S_1),\ \ C=\Cen_N(S_1),\ \ X=N/C.
\]
Then $X$ is almost simple with socle $S_1C/C\cong S_1$, and for every cyclic $p$-subgroup $H\leq S_1$, the normalizer $\Nor_X(HC/C)$ has a solvable supplement in $X$.
\end{lemma}

\begin{proof}
Since $C\cap S_1=\mathbf{Z}(S_1)=1$, we have $S_1C/C\cong S_1$. Then it follows from
\[
S_1C/C\trianglelefteq X\lesssim\Aut(S_1)
\]
that $X$ is almost simple with socle $S_1C/C\cong S_1$. To prove that $\Nor_X(HC/C)$ has a solvable supplement in $X$, we may assume without loss of generality that $H\neq1$. If $g\in\Nor_G(H)$, then
\[
  1\neq H=H^g\leq S_1\cap S_1^g.
\]
Conjugation by $G$ permutes the direct factors of $M$, and distinct direct
factors intersect trivially. Therefore, $S_1^g=S_1$ for each $g\in\Nor_G(H)$. This shows that $\Nor_G(H)\leq N$.

Since $G$ has property $\mathrm{SSN}_p$, there exists a solvable subgroup $K\leq G$ such that $G=\Nor_G(H)K$. Then since $\Nor_G(H)\leq N$, the Dedekind modular law gives
\[
  N=\Nor_G(H)(K\cap N).
\]
Moreover, $\Nor_G(H)C/C\leq\Nor_X(HC/C)$. This yields
\[
  X=\Nor_X(HC/C)\bigl((K\cap N)C/C\bigr),
\]
which means that $\Nor_X(HC/C)$ has a solvable supplement in $X$.
\end{proof}

\begin{proof}[\rm\textbf{Proof of Theorem~\ref{thm:main}}]
We argue by induction on $|G|$. Let $M$ be a minimal normal subgroup of $G$. By Lemma~\ref{lem:quotient}, $G/M$ has property $\mathrm{SSN}_2$. Hence, by induction, every nonabelian composition factor of $G/M$ is isomorphic to $\PSL_2(q)$ for some prime power $q\geq7$ with $q\equiv3\pmod4$. If $M$ is abelian, all composition factors contributed by $M$ are cyclic of
prime order. Otherwise
\[
  M=S_1\times\cdots\times S_r
\]
for pairwise isomorphic nonabelian simple groups $S_1,\ldots,S_r$. In this case, letting
\[
  X=\Nor_G(S_1)/\Cen_{\Nor_G(S_1)}(S_1),
\]
Lemma~\ref{lem:minimal-normal} asserts that $X$ is almost simple with socle isomorphic to $S_1$, and the normalizer in $X$ of every nontrivial cyclic $2$-subgroup of its socle has a solvable supplement. Note that $\Nor_X(\langle t\rangle)=\Cen_X(t)$ for every involution $t$ in the socle of $X$. Hence Proposition~\ref{prop:almost} gives $S_1\cong\PSL_2(q)$ with $q\geq7$ and $q\equiv3\pmod4$.
Therefore, in either case, all nonabelian composition factors of $M$ have the required form, and induction completes the proof.
\end{proof}

\begin{proof}[\rm\textbf{Proof of Corollary~\ref{cor:solvability}}]
Clearly,~\ref{item:solvability-solvable} implies~\ref{item:solvability-section}, and~\ref{item:solvability-section} implies~\ref{item:solvability-composition-factor}. Moreover, since $G$ has property $\mathrm{SSN}_2$, we conclude from Theorem~\ref{thm:main} that~\ref{item:solvability-composition-factor} implies~\ref{item:solvability-solvable}.
\end{proof}

\section*{Acknowledgements}

The first author was supported by the National Natural Science Foundation of China (grant no.~12471015) and the Natural Science Foundation of Guangdong Province (grant no.~2025A1515012321).

\end{document}